\documentclass[fleqn]{amsart}

\usepackage{relsize}
\usepackage{xcolor}

\newtheorem{theorem}{Theorem}
\newtheorem*{theorem*}{Theorem}
\newtheorem{lemma}[theorem]{Lemma}

\theoremstyle{definition}
\newtheorem*{definition*}{Definition}
\newtheorem{proposition}[theorem]{Proposition}

\newtheorem*{exmp*}{Example}

\newcommand{\N}{{\mathbb{N}}}

\begin{document}

\title[Completely uniformly distributed sequences]{Completely uniformly distributed sequences \\based on de {B}ruijn sequences}

\author{Emilio Almansi}
\address{Departamento de  Computaci\'on, Facultad de Ciencias Exactas y Naturales, Universidad de Buenos Aires, Argentina}
\email{ealmansi@gmail.com}

\author{Ver{\'o}nica\  Becher}
\address{Departamento de  Computaci\'on, Facultad de Ciencias Exactas y Naturales, Universidad de Buenos Aires \& ICC CONICET, Argentina}
\email{vbecher@dc.uba.ar}
\thanks{Supported by Universidad de Buenos Aires and  CONICET, Argentina.}



\date{September 24, 2019}

\begin{abstract}
  We study a construction published by Donald Knuth in 1965 
yielding a completely uniformly distributed sequence of real numbers. 
Knuth's work is based on de Bruijn sequences of increasing orders and alphabet sizes, 
which grow exponentially in each of the successive segments composing the generated sequence. 
In this work we present a similar albeit simpler construction using linearly increasing alphabet sizes, 
and give an elementary proof showing that the sequence it yields is also completely uniformly distributed. 
In addition, we present an alternative proof of the same result based on Weyl's criterion.
\end{abstract}

\maketitle

\section{Introduction}
Let $(\bar{x}_i)_{i \ge 1}$ be an infinite sequence of $k$-dimensional vectors of real numbers. For a given set $I \subseteq [0, 1)^k$, we denote by $\mathcal{A}(\bar{x}_1, \dots, \bar{x}_N ; I)$ the number of indexes $i \in [1,N]$ such that $\bar{x}_i$ lies in~$I$. A sequence $(\bar{x}_i)_{i \geq 1}$ of $k$-dimensional vectors in the unit cube is uniformly distributed if for every set $I$ included in $[0, 1)^k$,
\begin{equation*}
  \lim_{N\to\infty} \frac{1}{N} \mathcal{A}(\bar{x}_1, \dots, \bar{x}_N ; I) = |I| \text{,}
\end{equation*}
where $|I|$ is the measure of $I$.

For the theory of uniform distribution see the monograph  of 
 Kuipers and Niederreiter~\cite{KN}. Starting with a sequence $(x_i)_{i \geq 1}$ 
of real numbers in the unit interval, we define for every integer $k \geq 1$ 
the $k$-dimensional sequence ${\textstyle (\bar{x}^{(k)}_i)_{i \geq 1}}$ 
where ${\textstyle \bar{x}^{(k)}_i = (x_{i}, \ldots,x_{i + k - 1})}$ for $i \geq 1$. 
The sequence $(x_i)_{i \geq 1}$ is said to be completely uniformly distributed if for every integer 
$k \geq  1$, the sequence $(\bar{x}^{(k)} )_{i \geq 1}$ is uniformly distributed the $k$-dimensional unit cube.

If a sequence is  completely uniformly distributed then it also satisfies many other statistical properties common to all random sequences; see  the work of  Franklin~\cite{franklin-1963}. For example, for every fixed $k$ the probability of~$k$ consecutive terms having any specific relative order is $1/k!$.

The first reference to the concept of complete uniform distribution is the work of Korobov in 1948~\cite{Ko1,K} where he proves that the sequence $(f_\alpha(n) \mod 1)_{n\geq 1}$ where
\begin{equation*}
  f_\alpha(z) = \sum_{k\geq 1} e^{−k^\alpha} z^k, \mbox{with $\alpha\in(1,3)$}
\end{equation*}
is completely uniformly distributed. Since then, there have been many constructions of completely uniformly distributed sequences; see the books~\cite{KN,NW} and the references cited there. In~\cite{Levin99} Levin gives a   lower bound of decay of discrepancy of  completely uniformly distributed sequences and  gives constructions, using nets obtained with Sobol matrices and using Halton sequences, which achieve this low discrepancy bound.

Independently of the school of Korobov, 
in 1965 Knuth~\cite{knuth-1965} gives a construction of 
a completely uniformly distributed sequence $(x_i)_{i \geq 1}$ 
of real numbers based on de Bruijn sequences of increasing orders
 and alphabet sizes. In each of the successive segments composing 
the sequence, the alphabet sizes of the underlying de Bruijn 
sequences increase exponentially. 
Knuth gives an elementary proof that the sequence $(x_i)_{i \geq 1}$ 
is completely uniformly distributed  by reasoning about the  binary 
representation of the rational numbers $x_i$.
These binary representations are immediate 
 because the  alphabet size is always a power of $2$.

In the present work we give a similar albeit simpler construction than 
the one given by Knuth in~\cite{knuth-1965}
while preserving the property of complete uniform distribution. 
Our construction is based on de 
Bruijn sequences of linearly increasing orders
 and linearly increasing  alphabet sizes.
The argument used by Knuth to prove  that his sequence is completely uniformly distributed  does not hold in our case. 
We give  two proofs: one is elementary, and the other is based on Weyl's criterion.
The construction and the proofs we present here can
be generalized in a straightforward manner to other growths of alphabet sizes.
A first version of this work appears in~\cite{almansi-becher-2019}.

It may be possible to construct completely uniformly 
distributed sequences with  low discrepancy
 by using variants of de Bruijn sequences such as those considered by   
Becher and Carton in~\cite{necklaces} (nested perfect necklaces).
In that work, they show that 
 the binary expansion of the real number 
 defined by  Levin in~\cite{L99} 
 using Pascal matrices
is indeed the concatenation of such variants of de Bruijn sequences.

We finish this section with a short presentation of de Bruijn sequences. 
In Section~\ref{section:knuth-sequence} we present Knuth's sequence, and 
we devote  Section~\ref{sec:main-result} to  our construction. 

\subsection{De Bruijn Sequences}
A presentation of de Bruijn sequences with a historical background can be read from~\cite{berstel-2007}. 
A $b$-ary de Bruijn sequence of order~$k$ is a sequence of length~$b^k$ which, 
when viewed as a cycle, contains every possible $b$-ary sequence of length~$k$ exactly once as a contiguous subsequence.
For example, listed next are two distinct binary de Bruijn sequences of order~$3$:
\begin{equation*}
  \begin{aligned}
    0, 0, 0, 1, 0, 1, 1, 1; \\
    0, 0, 0, 1, 1, 1, 0, 1.
  \end{aligned}
\end{equation*}
Every  binary sequence of length $3$ appears exactly once as a contiguous subsequence in each example. 
This includes those instances such as $1, 0, 0$ which wrap around the right-hand end of the sequences.

Each $b$-ary de Bruijn sequence of order~$k$  is the label of a Hamiltonian  cycle in the 
directed graph $G_{b,k}=(V,E)$
where $V$ is the set of all $b$-ary sequences of length~$k$ and 
$E=\{(u,v): u=a_1, \dots, a_k ,\;\; v=a_2, \dots, a_k, c ,\;\; c\in \{0, \dots b-1\}\}$.
Due to good properties of de Bruijn graphs, each $b$-ary de Bruijn sequence of order~$k$ can also be identified with the label of an Eulerian cycle in $G_{b,k-1}$.
By the BEST theorem there are exactly $(b!)^{b^{k - 1}} / b^k$ different $b$-ary de Bruijn sequences of order~$k$.

A $b$-ary Ford sequence of order $k$, denoted by $F^{(b, k)}$, 
is the lexicographically least $b$-ary de Bruijn sequence of order~$k$. 
In the example above, the first sequence 
is the binary Ford sequence of order $3$, or $F^{(2, 3)} = 0, 0, 0, 1, 0, 1, 1, 1$.
In~\cite{fredricksen-1978}, Fredricksen and Maiorana introduce an algorithm for generating the
$b$-ary Ford sequence of order~$k$,
which  has constant amortized running time, see~\cite{ruskey-1991}.

\section{Knuth's Sequence}\label{section:knuth-sequence}

We follow Knuth's presentation in~\cite{knuth-1965}.
Knuth's sequence can be defined from  any given family of de Bruijn sequences. 
Here, we use Ford sequences because they are conveniently defined and because they can be generated efficiently.

\begin{definition*}
  Given $n \in \N$, let $F^{(2^n, n)} = f_1, \dots, f_{2^{n^2}}$ denote the $2^n$-ary Ford sequence of order~$n$. 
An $A$-sequence of order $n$, denoted by $A^{(n)}$, is the finite sequence of rational numbers obtained by dividing every term in $F^{(2^n, n)}$ by $2^n$:
  \begin{equation*}
    A^{(n)} = \frac{f_1}{2^n}, \frac{f_2}{2^n}, \dots, \frac{f_{2^{n^2}}}{2^n} = \bigg( \frac{f_i}{2^n} \bigg)_{i = 1,..,2^{n^2}}.
  \end{equation*}
 A $B$-sequence of order $n$, denoted by $B^{(n)}$, is the sequence obtained from the concatenation of $n 2^{2 n}$ copies of $A^{(n)}$:
  \begin{equation*}
    B^{(n)} = \left< \underbrace{A^{(n)} ; A^{(n)} ; \dots ; A^{(n)}}_{n 2^{2 n} \text{ times}} \right> \text{.}
  \end{equation*}
\end{definition*}

By construction, the size of $A^{(n)}$ is 
\[
|A^{(n)}| = |F^{(2^n, n)}| = 2^{n^2}
\]
and the size of $B^{(n)}$ is 
\[
|B^{(n)}| = n 2^{2 n} |A^{(n)}| = n 2^{2 n} 2^{n^2}.
\] 
Notice that, for any given $n$, all terms in $A^{(n)}$ and in $B^{(n)}$ are numbers
 in the set ${\textstyle{ \left \{ 0, \frac{1}{2^n}, \frac{2}{2^n}, \dots, \frac{2^n - 1}{2^n} \right \} \subset [0, 1)}}$.
For example, when $n = 2$:
\begin{equation*}
  \begin{aligned}
    & F^{(4, 2)} = 0, 0, 1, 0, 2, 0, 3, 1, 1, 2, 1, 3, 2, 2, 3, 3 \\
    & A^{(2)} = \frac{0}{4}, \frac{0}{4}, \frac{1}{4}, \frac{0}{4}, \frac{2}{4}, \frac{0}{4}, \frac{3}{4}, \frac{1}{4}, \frac{1}{4}, \frac{2}{4}, \frac{1}{4}, \frac{3}{4}, \frac{2}{4}, \frac{2}{4}, \frac{3}{4}, \frac{3}{4} \\
    & B^{(2)} = \left< \underbrace{A^{(2)} ; \dots ; A^{(2)}}_{32 \text{ times}} \right> = \underbrace{\frac{0}{4}, \frac{0}{4}, \dots, \frac{3}{4}, \frac{3}{4}}_{A^{(2)}}, \dots, \underbrace{\frac{0}{4}, \frac{0}{4}, \dots, \frac{3}{4}, \frac{3}{4}}_{A^{(2)}}
  \end{aligned}
\end{equation*}
and $|A^{(2)}| = 16$, $|B^{(2)}| = 512$.

\begin{definition*}
  Knuth's sequence, denoted by $K$, is the infinite sequence of rational 
numbers resulting from the concatenation of the sequences $B^{(n)}$ for $n = 1, 2, \dots$:
  \begin{equation*}
    K = \left< B^{(1)} ; B^{(2)} ;  B^{(3)} ; \dots \right> \text{.}
  \end{equation*}
\end{definition*}

\begin{theorem*}[Knuth 1965,~\cite{knuth-1965}, page 268]
  The sequence $K$ is completely uniformly distributed.
\end{theorem*}

Knuth provides an elementary proof of this theorem. 
Two choices in Knuth's construction 
play an important role in his proof. 
One is that  the number of repetitions of each $A$-sequence within a $B$-sequence is $n 2^{2n}$. 
It follows from Knuth's proof that to ensure complete uniform distribution 
it is sufficient that the number of repetitions of each $A$-sequence within a $B$-sequence
grows faster than~$2^{2n}$.   
For instance,  $\lceil \log n \rceil 2^{2n}$ suffices as well.

The other choice in Knuth's construction is that the  alphabet sizes of the Ford sequences grow exponentially as~$2^n$.
This allows us to reason about the rational numbers comprised in the sequence $K$ 
in terms of their binary representations.  Knuth considers
 the distribution of the $m$ most-significant bits of each of  the terms in a $2^n$-ary Ford sequence, 
where $m \le n$. See~\cite[Lemma 2]{knuth-1965}.

\section{Main result} \label{sec:main-result}

We now present the main contribution of this work, which is a variant of Knuth's construction 
 based on Ford sequences with linearly increasing alphabet sizes.

\begin{definition*}
  Given $n  \in \N$ and a function $t : \mathbb{N} \mapsto \mathbb{N}$, 
let $F^{(n, n)} = f_1, \dots, f_{n^n}$ denote the $n$-ary Ford sequence of order~$n$. 
A $C$-sequence of order $n$, denoted by $C^{(n)}$, is the finite sequence of rational numbers 
obtained by dividing each of the terms in $F^{(n, n)}$ by $n$:
  \begin{equation*}
    C^{(n)} = \frac{f_1}{n}, \frac{f_2}{n}, \dots, \frac{f_{n^n}}{n} = \bigg( \frac{f_i}{n} \bigg)_{i = 1,..,n^n} \text{,}
  \end{equation*}
  and a $D$-sequence of order $n$, denoted by $D^{(n, t)}$,
 is the sequence obtained from the concatenation of $t(n)$ consecutive copies of $C^{(n)}$:
  \begin{equation*}
    D^{(n, t)} = \left< \underbrace{C^{(n)} ; C^{(n)} ; \dots ; C^{(n)}}_{t(n) \text{ times}} \right> \text{.}
  \end{equation*}
\end{definition*}

The size of $C^{(n)}$ is 
\[
|C^{(n)}| = |F^{(n, n)}| = n^n,
\] 
and the size of $D^{(n, t)}$ is 
\[
|D^{(n, t)}| = t(n) |C^{(n)}| = t(n) n^n.
\]
 For any given $n$, all terms in $C^{(n)}$ and in $D^{(n, t)}$ 
are numbers in the set $\textstyle {\left \{ 0, \frac{1}{n}, \frac{2}{n}, \dots, \frac{n - 1}{n} \right \} \subset [0, 1)}$.
The key difference between the way $C$-sequences are constructed when compared to 
$A$-sequences from the previous section is that, as the order of the sequence grows, 
the alphabet size for the underlying Ford sequence grows linearly ($1, 2, 3, 4, \dots$) 
rather than exponentially ($2, 4, 8, 16, \dots$).
For example, when $n = 3$ and $t(n) = n$: %id(n) = n$:
\begin{equation*}
  \hspace*{-10mm}\begin{aligned}
    & F^{(3, 3)} = 0, 0, 0, 1, 0, 0, 2, 0, 1, 1, 0, 1, 2, 0, 2, 1, 0, 2, 2, 1, 1, 1, 2, 1, 2, 2, 2 \\
    & C^{(3)} = \frac{0}{3}, \frac{0}{3}, \frac{0}{3}, \frac{1}{3}, \frac{0}{3}, \frac{0}{3}, \frac{2}{3}, \frac{0}{3}, \frac{1}{3}, \frac{1}{3}, \frac{0}{3}, \frac{1}{3}, \frac{2}{3}, \frac{0}{3}, \frac{2}{3}, \frac{1}{3}, \frac{0}{3}, \frac{2}{3}, \frac{2}{3}, \frac{1}{3}, \frac{1}{3}, \frac{1}{3}, \frac{2}{3}, \frac{1}{3}, \frac{2}{3}, \frac{2}{3}, \frac{2}{3} \\
    & D^{(3, id)} = \left< C^{(3)} ; C^{(3)} ; C^{(3)} \right> = \underbrace{\frac{0}{3}, \frac{0}{3}, \dots, \frac{2}{3}, \frac{2}{3}}_{C^{(3)}}, \underbrace{\frac{0}{3}, \frac{0}{3}, \dots, \frac{2}{3}, \frac{2}{3}}_{C^{(3)}}, \underbrace{\frac{0}{3}, \frac{0}{3}, \dots, \frac{2}{3}, \frac{2}{3}}_{C^{(3)}}
  \end{aligned}
\end{equation*}
and $|C^{(3)}| = 27$, $|D^{(3, id)}| = 81$.

\begin{definition*}
  Given $t : \mathbb{N} \mapsto \mathbb{N}$, the sequence $L^{(t)}$ is the infinite sequence of real numbers obtained from the concatenation of the sequences $D^{(n, t)}$ for $n = 1, 2, \dots$:
  \begin{equation*}
    L^{(t)} = \left< D^{(1,t)} ; D^{(2,t)} ;  D^{(3,t)} ; \dots \right> \text{.}
  \end{equation*}
\end{definition*}

\begin{theorem}\label{thm:l-is-comp-unif-dist}
  If $t : \mathbb{N} \mapsto \mathbb{N}$ is a non-decreasing function and ${\textstyle{\lim_{n \to \infty} n / t(n) = 0}}$, then the sequence $L^{(t)}$ is completely uniformly distributed.
\end{theorem}

\begin{exmp*}
  If $t(n) = sq(n) = n^2$, then:
  \begin{equation*}
    L^{(sq)} = \left< C^{(1)} ; \underbrace{C^{(2)} ; \dots ;  C^{(2)}}_{4 \text{ copies}} ;  \underbrace{C^{(3)} ; \dots ;  C^{(3)}}_{9 \text{ copies}} ; \dots \right> \text{,}
  \end{equation*}
  and $L^{(sq)}$ is completely uniformly distributed.
\end{exmp*}

\subsection{Proof of Theorem 1}\label{subsection:proof-of-theorem-1}
Within the current section, let $t : \mathbb{N} \mapsto \mathbb{N}$ 
to be an arbitrary but fixed function. For convenience we write $D^{(n)}$ and $L$ in place of $D^{(n, t)}$ and $L^{(t)}$.
Consider a prefix of $L$ of length $N$, denoted by $L_{1 : N}$. 
Let $p, q, r \in \mathbb{N}$ be the integers determined by $N$ such that:
\begin{equation*}
  L_{1:N} = \left< D^{(1)} ; \dots ; D^{(r - 1)} ; \underbrace{C^{(r)} ; \dots ; C^{(r)}}_{q \text{ times}} ; C^{(r)}_{1:p} \right>
\end{equation*}
where $0 \le q < t(r)$ and $1 \le p \le r^r$.
Here, $r$ is the order of the rightmost, possibly incomplete $D$-sequence present in $L_{1:N}$.
The number $q$ is the amount of complete $C$-sequences of order $r$ appearing before the rightmost,
possibly incomplete $C$-sequence, while $p$ is the amount of terms in it.
%Note that the values of $p$, $q$ and $r$ are uniquely determined by the value of~$N$:
Thus,
\begin{equation}\label{equation:n-p-q-r}
  \begin{aligned}
    N & = \sum_{s = 1}^{r - 1} |D^{(s)}| + q |C^{(r)}| + p 
       = \sum_{s = 1}^{r - 1} t(s) s^s + q r^r + p \text{.}
  \end{aligned}
\end{equation}

Let $k$ be a positive integer and 
let $I = [u_1, v_1)  \times \dots \times [u_k, v_k)$ be a set such that $I \subseteq [0, 1)^k$.
%where both $k$ and $I$ have arbitrary but fixed values. 
Let $N$ range freely over the natural numbers, 
and let the quantity $\nu_N$ denote the number of windows of $L$ of size $k$ starting at indices 
$i = 1 \dots N$ that belong to the set $I$:
\begin{equation*}
  \nu_N = \mathcal{A}(\bar{w}_1, \dots, \bar{w}_N ; I) \text{,}
\end{equation*}
where $\bar{w}_i = (L_i, \dots, L_{i + k - 1})$ for $i = 1 \dots N$.

Consider sufficiently large values of $N$ such that $k < r$. This is always possible since $r$ is an unbounded, non-decreasing function of $N$. We can decompose $L_{1:N}$ into four consecutive sections; namely, sequences $S^{(1)}$, $S^{(2)}$, $S^{(3)}$ and $S^{(4)}$:
\begin{equation*}
  \begin{aligned}
    L_{1:N} & = \left< S^{(1)} ; S^{(2)} ; S^{(3)} ; S^{(4)} \right> \text{, \hspace*{5mm} where} \\
    S^{(1)} & = \left< D^{(1)} ; D^{(2)} ; \dots ; D^{(k - 1)} \right> \\
    S^{(2)} & = \left< D^{(k)} ; D^{(k + 1)} ; \dots ; D^{(r - 1)} \right> \\
    S^{(3)} & = \left< \underbrace{C^{(r)} ; \dots ; C^{(r)}}_{q \text{ times}} \right> \\ 
    S^{(4)} & = C^{(r)}_{1:p} \text{.}
  \end{aligned}
\end{equation*}

Notice that $S^{(1)}$ and $S^{(3)}$ can potentially be empty, such as when $k = 1$ or $q = 0$, respectively.

We denote the cumulative sums of the sizes of the sequences defined above as 
$n_0 = 0$, and $n_j = n_{j - 1} + |S^{(j)}|$ for $j = 1, 2, 3, 4$. 
Now, we can similarly decompose $\nu_N$ into five parts. 
If we let 
\[
\bar{w}_i = (L_i, \dots, L_{i + k - 1})
\]
 for $i = 1 \dots N$, then:
\begin{equation*}
  \begin{aligned}
    \nu_N & = \nu^{(1)}_N + \nu^{(2)}_N + \nu^{(3)}_N + \nu^{(4)}_N + \varepsilon_b \text{, \hspace*{5mm} where} \\
    \nu^{(j)}_N & = \mathcal{A}(\bar{w}_{1 + n_{j - 1}}, \dots, \bar{w}_{n_j - k + 1} ; I)
  \end{aligned}
\end{equation*}
for some $\varepsilon_b \le 3 (k - 1)$.

For each $j = 1, 2, 3, 4$, the quantity $\nu^{(j)}_N$
 accounts for windows contained entirely within the sequence $S^{(j)}$, and $\varepsilon_b$
 accounts for all windows crossing any of the three borders between the four sections. 
This is enough to account for all possible windows, since any given window is either 
entirely contained in some section, or it starts at a given section and ends at a subsequent one, 
thereby crossing a border.

We can rewrite the value of each $\nu^{(j)}_N$ in terms of windows over each sequence~$S^{(j)}$.
 If we let 
\[
\bar{s}_i^{(j)} = (S_i^{(j)}, \dots, S_{i + k - 1}^{(j)})
\]
 for $j = 1, 2, 3, 4$ and $i = 1, \dots, |S^{(j)}| - k + 1$, then:
\begin{equation}\label{equation:nu-n-split-five-parts}
  \nu^{(j)}_N = \mathcal{A}(\bar{s}_1^{(j)}, \dots, \bar{s}_{|S^{(j)}| - k + 1}^{(j)} ; I) \text{.}
\end{equation}

Before obtaining more precise expressions for these quantities, we  state the following three technical propositions.

\begin{proposition}\label{proposition:inequality-sandwich}
  If $n \in \mathbb{N}$ and $x, y \in \mathbb{R}$ such that $[x, y) \subseteq [0, n)$, then the number of integers from the set $\left \{ 0, 1, \dots, n - 1 \right \}$ contained in $\left[x, y\right)$ is equal to $y - x + \varepsilon$ for some $\varepsilon \in (-1, 1)$.
\end{proposition}

\begin{proof}
  Since $0 \le y$, there are exactly $\lceil y \rceil = y + \varepsilon_y$
 non-negative integers in the set $[0, y)$ for some $\varepsilon_y \in [0, 1)$. 
Similarly for $x$, there are exactly $\lceil x \rceil = x + \varepsilon_x$ 
non-negative integers in the set $[0, x)$ for some $\varepsilon_x \in [0, 1)$. 
The difference between these two quantities is equal to the number of non-negative
 integers contained in the set $[x, y)$, which is $y - x + (\varepsilon_y - \varepsilon_x)$. 
Observing that $(\varepsilon_y - \varepsilon_x) \in (-1, 1)$, and that all non-negative 
integers between $x$ and $y$ belong to the set $\left \{ 0, 1, \dots, n - 1 \right \}$, the proof is complete.
\end{proof}

\begin{proposition}\label{proposition:product-of-sums}
  If $k$ is a positive integer and  $a_1, a_2, \dots, a_k$ and $b_1, b_2, \dots, b_k$ are two sequences 
of real numbers of length~$k$, then the product of their element-by-element sums 
can be expressed as follows:
  \begin{equation*}
    \prod_{d = 1}^{k} a_d + b_d = \prod_{d = 1}^{k} a_d + \mathlarger{\sum}_{j = 1}^{2^k - 1} \left [ \mathlarger{\prod}_{d = 1}^{k}
      \left \{ \begin{array}{lr}
        a_d & \left\lfloor \frac{j}{2^{d - 1}} \right\rfloor \text{is even} \\
        b_d & \text{otherwise}
      \end{array} \right \} \right ] \text{.}
  \end{equation*}
\end{proposition}

\begin{proof}
  By induction on $k$. First, notice that the property holds for~$k = 1$:
  \begin{equation*}
    \begin{aligned}
      & \prod_{d = 1}^{1} a_d + b_d = a_1 + b_1 \text{, and} \\
      & \prod_{d = 1}^{1} a_d + \mathlarger{\sum}_{j = 1}^{1} \left [ \mathlarger{\prod}_{d = 1}^{1}
      \left \{ \begin{array}{lr}
        a_d & \left\lfloor \frac{j}{2^{d - 1}} \right\rfloor \text{is even} \\
        b_d & \text{otherwise}
      \end{array} \right \} \right ] = a_1 + b_1 \text{.}
    \end{aligned}
  \end{equation*}
  Next, we see that the inductive step holds for any $k$. First,

  \begin{equation*}
    \begin{aligned}
      \hspace*{-5mm}\prod_{d = 1}^{k + 1} a_d + b_d
        & = (a_{k + 1} + b_{k + 1}) \prod_{d = 1}^{k} a_d + b_d \text{, \hspace*{5mm} and by I. H.} \\
        & = (a_{k + 1} + b_{k + 1}) \left[ \prod_{d = 1}^{k} a_d + \mathlarger{\sum}_{j = 1}^{2^k - 1} \left [ \mathlarger{\prod}_{d = 1}^{k}
      \left \{ \begin{array}{lr}
        a_d & \left\lfloor \frac{j}{2^{d - 1}} \right\rfloor \text{is even} \\
        b_d & \text{otherwise}
      \end{array} \right \} \right ] \right] \\
        & = \prod_{d = 1}^{k + 1} a_d + a_{k + 1} \mathlarger{\sum}_{j = 1}^{2^k - 1} \left [ \mathlarger{\prod}_{d = 1}^{k}
        \left \{ \begin{array}{lr}
          a_d & \left\lfloor \frac{j}{2^{d - 1}} \right\rfloor \text{is even} \\
          b_d & \text{otherwise}
        \end{array} \right \} \right ] \\
        & \hspace*{4mm} + b_{k + 1} \prod_{d = 1}^{k} a_d + b_{k + 1} \mathlarger{\sum}_{j = 1}^{2^k - 1} \left [ \mathlarger{\prod}_{d = 1}^{k}
        \left \{ \begin{array}{lr}
          a_d & \left\lfloor \frac{j}{2^{d - 1}} \right\rfloor \text{is even} \\
          b_d & \text{otherwise}
        \end{array} \right \} \right ] \text{.}
    \end{aligned}
  \end{equation*}

Since for every $j = 1 \dots 2^{k} - 1$ the value ${\textstyle \left\lfloor \frac{j}{2^{k}} \right\rfloor = 0}$ 
and is therefore even, we can add the factor $a_{k + 1}$ to the product in 
the second term simply by raising the upper limit to $k + 1$. 
Similarly, in the fourth term we can add the factor $b_{k + 1}$ to the 
product by raising the upper limit to $k + 1$ and changing the limits in the 
sum to $j = (2^k + 1) \dots (2^k + 2^k - 1)$. This is true because adding 
$2^k$ to $j$ does not change the value of 
${\textstyle \left\lfloor \frac{j}{2^{d - 1}} \right\rfloor}$ for any $d \le k$, 
but when $d = k + 1$ the value ${\textstyle \left\lfloor \frac{j}{2^{k}} \right\rfloor = 1}$ 
and is therefore odd. The third term can be rewritten as a similar product 
for a value of $j = 2^k$, and substituting into the equation above:
  \begin{equation*}
    \begin{aligned}
      \prod_{d = 1}^{k + 1} a_d + b_d
        & = \prod_{d = 1}^{k + 1} a_d + \mathlarger{\sum}_{j = 1}^{2^k - 1} \left [ \mathlarger{\prod}_{d = 1}^{k + 1}
        \left \{ \begin{array}{lr}
          a_d & \left\lfloor \frac{j}{2^{d - 1}} \right\rfloor \text{is even} \\
          b_d & \text{otherwise}
        \end{array} \right \} \right ] \\
        & \hspace*{4mm} + \mathlarger{\sum}_{j = 2^k}^{2^k} \left [ \mathlarger{\prod}_{d = 1}^{k + 1}
        \left \{ \begin{array}{lr}
          a_d & \left\lfloor \frac{j}{2^{d - 1}} \right\rfloor \text{is even} \\
          b_d & \text{otherwise}
        \end{array} \right \} \right ] \\
        & \hspace*{4mm} + \mathlarger{\sum}_{j = 2^k + 1}^{2^k + 2^k - 1} \left [ \mathlarger{\prod}_{d = 1}^{k}
        \left \{ \begin{array}{lr}
          a_d & \left\lfloor \frac{j}{2^{d - 1}} \right\rfloor \text{is even} \\
          b_d & \text{otherwise}
        \end{array} \right \} \right ] \\
        & = \prod_{d = 1}^{k + 1} a_d + \mathlarger{\sum}_{j = 1}^{2^{k + 1} - 1} \left [ \mathlarger{\prod}_{d = 1}^{k + 1}
        \left \{ \begin{array}{lr}
          a_d & \left\lfloor \frac{j}{2^{d - 1}} \right\rfloor \text{is even} \\
          b_d & \text{otherwise}
        \end{array} \right \} \right ] \text{,}
    \end{aligned}
  \end{equation*}
  which completes the proof.
\end{proof}

\begin{proposition}\label{proposition:sum-i-to-the-i-m-1}
  Given $n \in \mathbb{N}$, the following holds:
  \begin{equation*}
    \sum_{i = 1}^{n} i^{i - 1} \le 2 n^{n - 1} \text{.}
  \end{equation*}
\end{proposition}

\begin{proof}
  By induction on $n$. The property holds for $n = 1$ and $n = 2$:
  \begin{equation*}
    \sum_{i = 1}^{1} i^{i - 1} \le 2 \text{, } \sum_{i = 1}^{2} i^{i - 1} \le 4
  \end{equation*}
  and the inductive step holds for $n \ge 2$:
  \begin{equation*}
    \sum_{i = 1}^{n + 1} i^{i - 1} = \underbrace{\sum_{i = 1}^{n} i^{i - 1}}_{
      {
        \begin{aligned}
          & \text{\small $\le 2 n^{n - 1}$} \\
          & \text{\small by I. H.}
        \end{aligned}
      }%
    } + (n + 1)^{n} \le n n^{n - 1} + (n + 1)^{n} \le 2 (n + 1)^{n} \text{.}
  \end{equation*}

  Therefore, the property holds for all $n \in \mathbb{N}$.
\end{proof}

We now obtain an expression for the number of windows of a $C$-sequence which are contained in the set $I$. 
This is useful for evaluating $\nu_N$. %, as we shall seen later on.

\begin{lemma}\label{lemma:count-windows-c-sequence-cyclic}
  Given a positive integer $k$ and a set $I = [u_1, v_1) \times \dots \times [u_k, v_k)$ where $I \subseteq [0, 1)^k$, let $n \in \mathbb{N}$ such that $k \le n$ and consider the sequence $C^{(n)}$ as a cyclic sequence. If we let $\bar{c}_i = (C^{(n)}_i, \dots, C^{(n)}_{i + k - 1})$ for $i = 1, \dots, n^n$, then the number of windows of size $k$ from the sequence $C^{(n)}$ that lie in $I$ is:
  \begin{equation*}
    \mathcal{A}(\bar{c}_1, \dots, \bar{c}_{n^n} ; I) = n^n |I| + n^{n - 1} (2^k - 1) \varepsilon \text{.}
  \end{equation*}
  for some $\varepsilon \in (-1, 1)$.
\end{lemma}

\begin{proof}
  First, notice that any given window is contained in the set $I$ if and only if the following is true:
  \begin{equation*}
    \begin{aligned}
        \begin{array}{cccc}
          \bar{c}_i \in I \Longleftrightarrow & u_1 \le & C^{(n)}_{i}         & < v_1 \\
                              & & \vdots & \\
                              & u_k \le & C^{(n)}_{i + k - 1} & < v_k
        \end{array}
    \end{aligned}
  \end{equation*}
  where $i = 1 \dots n^n$ and indices are taken modulo $n^n$.

   Since all terms in $C^{(n)}$ are numbers in the set ${\textstyle \left \{ 0, \frac{1}{n}, \dots, \frac{n - 1}{n} \right \}}$, we multiply both sides of each inequality by $n$, allowing us to reason about integers belonging to a Ford sequence instead of rational numbers. We obtain the following:
  \begin{equation*}
    \begin{aligned}
        \begin{array}{cccc}
          \bar{c}_i \in I \Longleftrightarrow & n u_1 \le & F^{(n, n)}_{i}         & < n v_1 \\
                              & & \vdots & \\
                              & n u_k \le & F^{(n, n)}_{i + k - 1} & < n v_k \text{.}
        \end{array}
    \end{aligned}
  \end{equation*}
  According to  Proposition~\ref{proposition:inequality-sandwich}, for each inequality above with $d = 1 \dots k$ there are exactly $n v_d - n u_d + \varepsilon_d$ possible solutions in the set $ \left \{ 0, 1, \dots, n - 1 \right \} $ for some value $\varepsilon_d \in (-1, 1)$. 
This yields a total of ${\textstyle \prod_{d = 1}^{k} [n (v_d - u_d) + \varepsilon_d]}$ 
possible solutions to the system of inequalities. Each solution, when seen as an $n$-ary sequence of length~$k$, 
appears exactly $n^{n - k}$ times as a contiguous subsequence in $F^{(n, n)}$. 
This is true because there are  $n^{n - k}$ ways of extending an $n$-ary sequence of 
length $k$ to one of length $n$ and, by construction, each of these appears exactly once in $F^{(n, n)}$ 
when viewed as a cycle. Since $i$ ranges exactly once over each possible window of $F^{(n, n)}$, then:
  \begin{equation}\label{equation:count-windows-c-n}
    \begin{aligned}
      \mathcal{A}(\bar{c}_1, \dots, \bar{c}_{n^n} ; I) & = n^{n - k} \prod_{d = 1}^{k} [n (v_d - u_d) + \varepsilon_d] 
       = n^n \prod_{d = 1}^{k} [(v_d - u_d) + \varepsilon_d / n] \text{.}
    \end{aligned}
  \end{equation}
  Using Proposition~\ref{proposition:product-of-sums}, we can expand this into the following:
  \begin{equation}\label{equation:product-expansion}
    \begin{aligned}
      n^n \prod_{d = 1}^{k} [(v_d - u_d) + \varepsilon_d / n]
      & = n^n \prod_{d = 1}^{k} (v_d - u_d) \\
      & \hspace*{4mm} + n^n \mathlarger{\sum}_{j = 1}^{2^k - 1}
      \left [ \mathlarger{\prod}_{d = 1}^{k} \left \{ \begin{array}{lr}
      (v_d - u_d) & \left\lfloor \frac{j}{2^{d - 1}} \right\rfloor \text{is even} \\
      \varepsilon_d / n & \text{otherwise}
    \end{array} \right \} \right ] \text{.}
    \end{aligned}
  \end{equation}

  If we define $\varepsilon'_j$ for $j = 1 \dots 2^k - 1$ as:
  \begin{equation*}
    \varepsilon'_j / n = \mathlarger{\prod}_{d = 1}^{k} \left \{ \begin{array}{lr}
      (v_d - u_d) & \left\lfloor \frac{j}{2^{d - 1}} \right\rfloor \text{is even} \\
      \varepsilon_d / n & \text{otherwise}
    \end{array} \right \}
  \end{equation*}
  then, for each $j$, the value $\varepsilon'_j \in (-1, 1)$. This is true because the product on the right-hand side is composed of terms $(v_d - u_d) \in (-1, 1)$ and $\varepsilon_d / n \in (-1/n, 1/n)$ and, since $j > 0$, 
there is always at least one term of the second kind. 
Given that ${\textstyle{|I| = \prod_{d = 1}^{k} (v_d - u_d)}}$, 
we can further simplify equation~\eqref{equation:product-expansion} to get:
 
 \begin{equation}\label{equation:sum-of-epsilons}
    n^n \prod_{d = 1}^{k} [(v_d - u_d) + \varepsilon_d / n] = n^n |I| + n^n \sum_{j = 1}^{2^k - 1} \varepsilon'_j / n \text{.}
  \end{equation}
  Finally, since ${\textstyle - (2^k - 1) < \sum_{j = 1}^{2^k - 1} \varepsilon'_j < (2^k - 1)}$, there exists some $\varepsilon \in (-1, 1)$ such that:
  \begin{equation*}
    \mathlarger{\sum}_{j = 1}^{2^k - 1} \varepsilon'_j = (2^k - 1) \varepsilon
  \end{equation*}
  and hence, by~\eqref{equation:count-windows-c-n} and~\eqref{equation:sum-of-epsilons}:
  \begin{equation*}
    \mathcal{A}(\bar{c}_1, \dots, \bar{c}_{n^n} ; I) = n^n |I| + n^{n - 1} (2^k - 1) \varepsilon \text{.}
  \end{equation*}
\end{proof}

We can now give the proof of our main result.

\begin{proof}[Proof of Theorem 1]
 % Using Lemma~\ref{lemma:count-windows-c-sequence-cyclic}, we now prove the main result of this work.
 Let $k$ be a positive integer and let $I = [u_1, v_1) \times \dots \times [u_k, v_k)$ be
 a set such that $I \subseteq [0, 1)^k$.
Let $N$ range freely over the natural numbers. 
We now obtain an expression for $\nu_N / N$ and compute its limit when $N \to \infty$. 

Recall the following definitions:
  \begin{equation*}
    \begin{aligned}
      \nu^{(2)}_N & = \mathcal{A}(\bar{s}_1^{(2)}, \dots, \bar{s}_{|S^{(2)}| - k + 1}^{(2)} ; I) \\
      \nu^{(3)}_N & = \mathcal{A}(\bar{s}_1^{(3)}, \dots, \bar{s}_{|S^{(3)}| - k + 1}^{(3)} ; I) \\
      S^{(2)} & = \left< D^{(k)} ; D^{(k + 1)} ; \dots ; D^{(r - 1)} \right> \\
      S^{(3)} & = \left< \underbrace{C^{(r)} ; \dots ; C^{(r)}}_{q \text{ times}} \right> \text{.}
    \end{aligned}
  \end{equation*}
  where 
\[
\bar{s}_i^{(j)} = (S_i^{(j)}, \dots, S_{i + k - 1}^{(j)})
\]
 for $j = 2, 3$ and $ i = 1, \dots, |S^{(j)}| - k + 1$.

The sequences $S^{(2)}$ and $S^{(3)}$ are entirely composed of complete $C$-sequences 
of increasing orders which are larger than or equal to~$k$. Moreover, with the exception of the last, 
rightmost instance in each of $S^{(2)}$ and $S^{(3)}$, 
every single $C$-sequence is immediately succeeded by another $C$-sequence of the same or the following order, 
including those which are part of a $D$-sequence. 
Additionally, any window starting at the right-hand end of a $C$-sequence 
necessarily finishes within the first $k - 1$ elements of the following $C$-sequence, 
all of which are guaranteed to be~$0$.

  Therefore, the amount of windows of size $k$ contained in $I$ ranging over 
$S^{(2)}$ and $S^{(3)}$ is equal to the sum over each composing $C$-sequence \textit{viewed as a cycle}, 
with an error of at most $k - 1$ due to the fact that we are counting only windows 
entirely contained within each sequence. If we let 
\[
\bar{c}^{(s)}_i = (C^{(s)}_i, \dots, C^{(s)}_{i + k - 1})
\]
 for $s = k \dots r$ and $i = 1, \dots, s^s$, then:
  \begin{equation*}
    \begin{aligned}
      \nu^{(2)}_N & = \sum_{s = k}^{r - 1} \underbrace{\left[ t(s) \mathcal{A}(\bar{c}^{(s)}_1, \dots, \bar{c}^{(s)}_{s^s} ; I) \right]}_{\text{$C$-sequences contained in $D^{(s)}$}} + \varepsilon_{\nu^{(2)}_N} \\
      \nu^{(3)}_N & = q \mathcal{A}(\bar{c}^{(r)}_1, \dots, \bar{c}^{(r)}_{r^r} ; I)  + \varepsilon_{\nu^{(3)}_N}
    \end{aligned}
  \end{equation*}
  for some values $\varepsilon_{\nu^{(2)}_N} \le k - 1$, and $\varepsilon_{\nu^{(3)}_N} \le k - 1$.
  From Lemma~\ref{lemma:count-windows-c-sequence-cyclic}:
  \begin{equation*}
    \begin{aligned}
      \nu^{(2)}_N & = \sum_{s = k}^{r - 1} \left[ t(s) \left( s^s |I| + s^{s - 1} (2^k - 1) \varepsilon_s \right) \right] + \varepsilon_{\nu^{(2)}_N} \\
      \nu^{(3)}_N & = q \left( r^r |I| + r^{r - 1} (2^k - 1) \varepsilon_r \right)  + \varepsilon_{\nu^{(3)}_N}
    \end{aligned}
  \end{equation*}
  for some values of $\varepsilon_i \in (-1, 1)$, $i = k \dots r$.
  Substituting back into $\nu_N$ from equation~\eqref{equation:nu-n-split-five-parts}:
  \begin{equation*}
    \begin{aligned}
      \nu_N
        & = \nu^{(1)}_N \\
        & \hspace*{4mm} + \sum_{s = k}^{r - 1} \left[ t(s) \left( s^s |I| + s^{s - 1} (2^k - 1) \varepsilon_s \right) \right] + \varepsilon_{\nu^{(2)}_N} \\
        & \hspace*{4mm} + q \left( r^r |I| + r^{r - 1} (2^k - 1) \varepsilon_r \right)  + \varepsilon_{\nu^{(3)}_N} \\
        & \hspace*{4mm} + \nu^{(4)}_N + \varepsilon_b
    \end{aligned}
  \end{equation*}
  and factoring out terms multiplied by $|I|$, we get:

  \begin{equation*}
    \begin{aligned}
      \nu_N
        & = |I| \left[ \sum_{s = k}^{r - 1} t(s) s^s + q r^r \right] \\
        & \hspace*{4mm} + \nu^{(1)}_N \\
        & \hspace*{4mm} + \sum_{s = k}^{r - 1} \left[ t(s) s^{s - 1} (2^k - 1) \varepsilon_s \right] + \varepsilon_{\nu^{(2)}_N} \\
        & \hspace*{4mm} + q r^{r - 1} (2^k - 1) \varepsilon_r  + \varepsilon_{\nu^{(3)}_N} \\
        & \hspace*{4mm} + \nu^{(4)}_N + \varepsilon_b \text{.}
    \end{aligned}
  \end{equation*}
  We now rewrite the first term using the relationship between $p$, $r$, $q$ and $N$ from equation~\eqref{equation:n-p-q-r}:
  \begin{equation*}
    \begin{aligned}
      \nu_N
        & = |I| \left[ N - \sum_{s = 1}^{k - 1} t(s) s^s - p \right] \\
        & \hspace*{4mm} + \nu^{(1)}_N \\
        & \hspace*{4mm} + \sum_{s = k}^{r - 1} \left[ t(s) s^{s - 1} (2^k - 1) \varepsilon_s \right] + \varepsilon_{\nu^{(2)}_N} \\
        & \hspace*{4mm} + q r^{r - 1} (2^k - 1) \varepsilon_r  + \varepsilon_{\nu^{(3)}_N} \\
        & \hspace*{4mm} + \nu^{(4)}_N + \varepsilon_b.
    \end{aligned}
  \end{equation*}
  After dividing both sides by $N$ and rearranging terms we obtain:
  \begin{equation*}
    \begin{aligned}
      \frac{\nu_N}{N} - |I|
        & = \frac{p}{N} \left[ \frac{\nu^{(4)}_N}{p} - |I| \right] \\
        & \hspace*{4mm} + \frac{2^k - 1}{N} \left[ \sum_{s = k}^{r - 1} t(s) s^{s - 1} \varepsilon_s + q r^{r - 1} \varepsilon_r \right] \\
        & \hspace*{4mm} + \frac{1}{N} \left[ \nu^{(1)}_N - |I| \sum_{s = 1}^{k - 1} t(s) s^s + \varepsilon_{\nu^{(2)}_N} + \varepsilon_{\nu^{(3)}_N} + \varepsilon_b \right] \text{.}
    \end{aligned}
  \end{equation*}
  Taking limits on both sides as $N \to \infty$, the third term on the right-hand side approaches $0$ since the contents of the brackets are dependent on $k$ and bounded as a function of $N$. 
For  the second term, using the fact that $t$ is non-decreasing together 
with Proposition~\ref{proposition:sum-i-to-the-i-m-1}, we can see that:
 \[
      \frac{\sum\limits_{s = k}^{r - 1} t(s) s^{s - 1} \varepsilon_s}{N} \le \frac{t(r - 1) \sum\limits_{s = 1}^{r - 1} s^{s - 1}}{t(r - 1) (r - 1)^{(r - 1)}} \le \frac{2 (r - 1)^{(r - 2)}}{(r - 1)^{(r - 1)}} = \frac{2}{r - 1}
\]
and
\[
 \frac{q r^{r - 1} \varepsilon_r}{N} \le \frac{q r^{r - 1}}{q r^r} = \frac{1}{r}.
\]
Since $r$ is an unbounded, non-decreasing function of $N$, this  term approaches~$0$ as well.

  Finally, consider the first term on the right-hand side. 
 Since both ${\textstyle \frac{\nu^{(4)}_N}{p}}$ and $|I|$ are values between $0$ and $1$,
then ${\textstyle \left( \frac{\nu^{(4)}_N}{p} - |I| \right) \in [-1, 1]}$. 
Moreover, 
using the following identity:
  \begin{equation*}
    \frac{(x + 1)^{(x + 1)}}{x^x} = (x + 1) \left(1 + \frac{1}{x} \right)^x
  \end{equation*}
  with $x = r - 1$, 
and using that  $p \le r^r$ ,
we can see that for large values of $r$:
  \begin{equation*}
    \frac{p}{N} \le \frac{r^r}{t(r - 1) (r - 1)^{(r - 1)}} = \frac{r}{t(r - 1)} \left( 1 + \frac{1}{r - 1} \right)^{r - 1} \le \frac{r}{t(r - 1)} e
  \end{equation*}
  which by hypothesis also approaches $0$ as $N \to \infty$. Hence,
  \begin{equation*}
    \lim_{N \to \infty} \frac{\nu_N}{N} = |I|.
  \end{equation*}
Since $k$ and $I$ were chosen arbitrarily, the sequence $L$ is completely uniformly distributed and the proof of Theorem 1 is complete.
\end{proof}

\subsection{Alternative Proof of Theorem 1}

Weyl's criterion for the multidimensional case states that if $\bar{X} = \bar{x}_1, \bar{x}_2, \dots$ 
is a sequence of $k$-dimensional vectors of real numbers, then $\bar{X}$ 
is uniformly distributed in the unit cube if and only if for all non-zero $k$-dimensional 
vectors of integers $\bar{\ell} = (l_1, \dots, l_k)$:
\begin{equation*}
  \lim_{N \to \infty} \frac{1}{N} \sum_{n = 1}^{N} e^{2 \pi i \bar{\ell} \cdot \bar{x}_n} = 0 \text{,}
\end{equation*}
where $\bar{\ell} \cdot \bar{x}_n$ denotes the dot product of $\bar{\ell}$ and $\bar{x}_n$,
see~\cite{KN}.
Before applying Weyl's criterion to the sequence $L^{(t)}$, we first prove the following lemma
based on a well-known fact about the roots of unity.

\begin{lemma}\label{lemma:weyl-sum-over-cyclic-c-sequences}
  Given a positive integer $k$ and a non-zero $k$-dimensional vector of integers $\bar{\ell} = (l_1, \dots, l_k)$, 
let $n \in \mathbb{N}$ such that $n > \max\left(k, \min\left(|l_1|, \dots, |l_k|\right)\right)$ and consider the sequence $C^{(n)}$ as a cyclic sequence. If we let $\bar{c}_j = (C^{(n)}_j, \dots, C^{(n)}_{j + k - 1})$ for $j = 1, \dots, n^n$, then:
  \begin{equation}\label{equation:weyl-sum-over-cyclic-c-sequence}
    \sum_{j = 1}^{n^n} e^{2 \pi i \bar{\ell} \cdot \bar{c}_j} = 0.
  \end{equation}
\end{lemma}

\begin{proof}
  If we view $F^{(n, n)}$ as a cyclic sequence and let 
\[
\bar{f}_j = (F^{(n, n)}_j, \dots, F^{(n, n)}_{j + k - 1})
\]
 for $j = 1, \dots, n^n$, then from the definition of a 
$C$-sequence it follows that ${\textstyle \bar{c}_j = \frac{1}{n} \bar{f}_j}$.
  Then, if we define $\Gamma = \left \{ 0, 1, \dots, n - 1 \right \} $, 
every $\bar{\gamma} \in \Gamma^k$ appears exactly $n^{n - k}$ times as a contiguous subsequence of $F^{(n, n)}$.
 This is true because there are  $n^{n - k}$ 
ways of extending an $n$-ary sequence of length $k$ to one of length $n$ and, 
by construction, each of these appears exactly once in $F^{(n, n)}$ when viewed as a cycle.
  Substituting into the left-hand side of equation~\eqref{equation:weyl-sum-over-cyclic-c-sequence}:
  \begin{equation*}
    \sum_{j = 1}^{n^n} e^{2 \pi i \bar{\ell} \cdot (\frac{1}{n} \bar{f}_j)} = \sum_{j = 1}^{n^n} e^{\frac{2 \pi i}{n} \bar{\ell} \cdot \bar{f}_j} = n^{n - k} \sum_{\bar{\gamma} \in \Gamma^k} e^{\frac{2 \pi i}{n} \bar{\ell} \cdot \bar{\gamma}} \text{.}
  \end{equation*}
  Since $\bar{\ell} \cdot \bar{\gamma} \in \mathbb{Z}$ and the function $exp : \mathbb{Z} \to \mathbb{C}$, $exp(m) = e^{\frac{2 \pi i}{n} m}$ is periodic with a period equal to $n$, then:
  \begin{equation*}
    \sum_{\bar{\gamma} \in \Gamma^k} e^{\frac{2 \pi i}{n} \bar{\ell} \cdot \bar{\gamma}}
    = \sum_{r = 0}^{n - 1}
        \sum_{
          \substack{
            \bar{\gamma} \in \Gamma^k \\
            \bar{\ell} \cdot \bar{\gamma} \equiv r
          }
        } e^{\frac{2 \pi i}{n} r}
  \end{equation*}
  where the congruence $\bar{\ell} \cdot \bar{\gamma} \equiv r$ is taken modulo $n$.
  The conditions for the existence of solutions to equations of the form:
  \begin{equation*}
    \bar{\ell} \cdot \bar{\gamma} \equiv r \text{ (mod $n$) \hspace{5mm} where } \bar{\gamma} \in \Gamma^k
  \end{equation*}
  are well understood (see~\cite[page 114]{mccarthy-1986}). In particular, this equation only has solutions when $g = gcd(l_1, \dots, l_k, n)$ divides $r$ and, in such case, the total number of solutions is equal to $g n^{k - 1}$. Therefore:
  \begin{equation*}
    \begin{aligned}
      \sum_{r = 0}^{n - 1}
        \sum_{
          \substack{
            \bar{\gamma} \in \Gamma^k \\
            \bar{\ell} \cdot \bar{\gamma} \equiv r
          }
        } e^{\frac{2 \pi i}{n} r}
      & = \sum_{\substack{
        r = 0 \\
        g | r
      }}^{n - 1} g n^{k - 1} e^{\frac{2 \pi i}{n} r} \\
      & = g n^{k - 1} \sum_{\substack{
        r' = 0
      }}^{\left\lfloor \frac{n - 1}{g} \right\rfloor} e^{\frac{2 \pi i}{n} g r'}
    \end{aligned}
  \end{equation*}
  where the last step comes from substituting $r = g r'$. If we also substitute $n = g n'$ and observe that ${\textstyle{\left\lfloor \frac{n - 1}{g} \right\rfloor = n' - 1}}$:
  \begin{equation*}
    \begin{aligned}
      \sum_{\substack{
        r' = 0
      }}^{\left\lfloor \frac{n - 1}{g} \right\rfloor} e^{\frac{2 \pi i}{n} g r'}
      & = \sum_{\substack{
        r' = 0
      }}^{n' - 1} e^{\frac{2 \pi i}{n'} r'} \text{.}
    \end{aligned}
  \end{equation*}
  The term on the right-hand side is the sum of all the roots of unity of order $n'$.
 It is a well-known fact that this sum is equal to $0$ whenever $n' > 1$. 
Finally, since $n > \min(|l_1|, \dots, |l_k|) \ge g$, then $n' = n / g > 1$, and the proof is complete.
\end{proof}

\begin{proof}[Alternative Proof of Theorem 1]
Let $k$ be a positive integer and $\bar{\ell} = (l_1, \dots, l_k)$ 
a non-zero $k$-dimensional vector of integers.
%, where both $k$ and $\bar{\ell}$ have arbitrary but fixed values. 
Let $N$ range freely over the natural numbers. 
As in  Section~\ref{subsection:proof-of-theorem-1}, we consider a 
prefix of the sequence $L$ of length~$N$, denoted $L_{1:N}$, 
and we consider the values $p$, $q$, and $r$ as functions of $N$.
  Let $\bar{w}_i = (L_i, \dots, L_{i + k - 1})$ for $i = 1, 2, \dots$. 
Analogously to the first proof of Theorem 1, we define the complex value $\nu_N$ 
as the Weyl sum over the first $N$ windows of size~$k$ of~$L$:
  \begin{equation*}
    \nu_N = \sum_{j = 1}^{N} e^{2 \pi i \bar{\ell} \cdot \bar{w}_j} \text{.}
  \end{equation*}
  Let $m = \max\left(k, \min\left(|l_1|, \dots, |l_k|\right)\right)$ and consider sufficiently large values of $N$ such that $m < r$. 
As before, this is always possible since $r$ is an unbounded, non-decreasing function of $N$. We can decompose $L_{1:N}$ into four consecutive sections; namely, sequences $S^{(1)}$, $S^{(2)}$, $S^{(3)}$ and $S^{(4)}$:
  \begin{equation*}
    \begin{aligned}
      L_{1:N} & = \left< S^{(1)} ; S^{(2)} ; S^{(3)} ; S^{(4)} \right> \text{, \hspace*{5mm} where} \\
      S^{(1)} & = \left< D^{(1)} ; D^{(2)} ; \dots ; D^{(m - 1)} \right> \\
      S^{(2)} & = \left< D^{(m)} ; D^{(m + 1)} ; \dots ; D^{(r - 1)} \right> \\
      S^{(3)} & = \left< \underbrace{C^{(r)} ; \dots ; C^{(r)}}_{q \text{ times}} \right> \\ 
      S^{(4)} & = C^{(r)}_{1:p}
    \end{aligned}
  \end{equation*}
  and define $\nu_N^{(1)}$, $\nu_N^{(2)}$, $\nu_N^{(3)}$ and $\nu_N^{(4)}$ as the Weyl sums over windows entirely contained within each respective section, such that:
  \begin{equation*}
    \begin{aligned}
      \nu_N & = \nu^{(1)}_N + \nu^{(2)}_N + \nu^{(3)}_N + \nu^{(4)}_N + \varepsilon_b
    \end{aligned}
  \end{equation*}
  for some complex number $\varepsilon_b$ with $|\varepsilon_b| \le 3 (k - 1)$
 which accounts for all windows crossing over any border.

  Following the same reasoning as in Section~\ref{subsection:proof-of-theorem-1}, 
the values of $\nu_N^{(2)}$ and $\nu_N^{(3)}$ can be computed as the Weyl
 sums over the $C$-sequences composing $S^{(2)}$ and $S^{(3)}$ viewed as cycles, 
plus two error terms accounting for right borders. However, in this case, due to 
Lemma~\ref{lemma:weyl-sum-over-cyclic-c-sequences} and the fact that all $C$ 
sequences in $S^{(2)}$ and $S^{(3)}$ have orders greater than $m$, 
these sums vanish to zero. Therefore, 
\[
\nu_N = \nu^{(1)}_N + \varepsilon_{\nu^{(2)}_N} + \varepsilon_{\nu^{(3)}_N} + \nu^{(4)}_N + \varepsilon_b\]
for some complex values 
$\varepsilon_{\nu^{(2)}_N}, \varepsilon_{\nu^{(3)}_N}$ with $|\varepsilon_{\nu^{(2)}_N}|, |\varepsilon_{\nu^{(3)}_N}| \le k - 1$.
  We now consider the limit of $\nu_N / N$ as $N \to \infty$. Observe  that:
  \begin{equation*}
    \begin{aligned}
      \left| \frac{\nu_N}{N} \right|
        & \le \frac{1}{N} \left| \nu^{(1)}_N + \varepsilon_{\nu^{(2)}_N} + \varepsilon_{\nu^{(3)}_N} + \varepsilon_b \right| + \frac{1}{N} \left| \nu^{(4)}_N \right| \\
        & \le \frac{1}{N} \left[ \left| \nu^{(1)}_N \right| + | \varepsilon_{\nu^{(2)}_N} | + | \varepsilon_{\nu^{(3)}_N} | + \left| \varepsilon_b \right| \right] + \frac{p}{N} \\
        & \le \frac{1}{N} \left[ \left| \nu^{(1)}_N \right| + (k - 1) + (k - 1) + 3 (k - 1) \right] + \frac{p}{N} \text{.}
    \end{aligned}
  \end{equation*}
  The numerator in the first term is dependent on $k$ and $m$, and
 constant as a function of $N$. As shown before, the second term 
approaches $0$ as $r \to \infty$. Therefore, the sum of both terms 
approaches $0$ as $N \to \infty$, which in turn implies that $\nu_N / N$ 
vanishes as well.

  Given that $k$ and $\bar{\ell}$ were chosen arbitrarily, Weyl's criterion 
is satisfied for all values of $k$ and the sequence $L$ is completely 
uniformly distributed.
\end{proof}

\bibliographystyle{plain}
\bibliography{ab}

\begin{thebibliography}{10}

\bibitem{almansi-becher-2019}
E.~Almansi.
\newblock Completely equidistributed sequences based on {D}e {B}ruijn
  sequences.
\newblock {T}esis de {L}icenciatura. {D}irector: {V}. {B}echer, Universidad de
  Buenos Aires, Argentina, September 2019.

\bibitem{necklaces}
V.~Becher and O.~Carton.
\newblock Normal numbers and nested perfect necklaces.
\newblock {\em Journal of Complexity}, 54, 2019.

\bibitem{berstel-2007}
J.~Berstel and D.~Perrin.
\newblock The origins of combinatorics on words.
\newblock {\em European Journal of Combinatorics}, 28(3):996 -- 1022, 2007.

\bibitem{franklin-1963}
J.~N. Franklin.
\newblock Deterministic simulation of random processes.
\newblock {\em Mathematics of Computation}, 17(81):28--59, 1963.

\bibitem{fredricksen-1978}
H.~Fredricksen and J.~Maiorana.
\newblock Necklaces of beads in $k$ colors and $k$-ary de {B}ruijn sequences.
\newblock {\em Discrete Mathematics}, 23(3):207 -- 210, 1978.

\bibitem{knuth-1965}
D.~E. Knuth.
\newblock Construction of a random sequence.
\newblock {\em BIT Numerical Mathematics}, 5(4):246--250, 1965.

\bibitem{Ko1}
N.M. Korobov.
\newblock On functions with uniformly distributed fractional parts ({R}ussian).
\newblock {\em Doklady Akad. Nauk SSSR (N.S.)}, 62:21--22, 1948.

\bibitem{K}
N.M. Korobov.
\newblock {\em Exponential Sums and their Applications. Mathematics and Its
  Applications 80}.
\newblock Kluwer, 1992.

\bibitem{KN}
L.~Kuipers and H.~Niederreiter.
\newblock {\em Uniform Distribution of Sequences}.
\newblock Pure and Appl. Math. John Wiley, New York, 1974.

\bibitem{Levin99}
M.~B. Levin.
\newblock Discrepancy estimates of completely uniformly distributed and
  pseudorandom number sequences.
\newblock {\em International Mathematics Research Notices},
  1999(22):1231--1251, 1999.

\bibitem{L99}
M.~B. Levin.
\newblock On the discrepancy estimate of normal numbers.
\newblock {\em Acta Arithmetica}, 88(2):99--111, 1999.

\bibitem{mccarthy-1986}
P.J. McCarthy.
\newblock {\em Introduction to {A}rithmetical {F}unctions}.
\newblock Universitext Series. Springer-Verlag, 1986.

\bibitem{NW}
Harald Niederreiter and Arne Winterhof.
\newblock {\em Applied Number Theory}.
\newblock Springer International Publishing, 2015.

\bibitem{ruskey-1991}
F.~Ruskey, C.~Savage, and T.~M.~Y. Wang.
\newblock Generating necklaces.
\newblock {\em Journal of Algorithms}, 13(3):414 -- 430, 1992.

\end{thebibliography}

\end{document}